\documentclass[12pt,psamsfonts]{amsart}

\newtheorem{theorem}{Theorem}[section]
\newtheorem{lemma}[theorem]{Lemma}

\theoremstyle{definition}

\theoremstyle{remark}

\numberwithin{equation}{section}

\usepackage{graphicx}
\usepackage{amssymb}
\usepackage{amscd}
\usepackage{amsmath}
\usepackage{amsfonts}
\usepackage{amsbsy}
\usepackage{epsfig}

\begin{document}

\title[Liouvillian integrability of differential systems]{Liouvillian integrability of polynomial differential systems}



\author[X. Zhang]{Xiang Zhang}
\address{Department of Mathematics, and MOE-LSC, Shanghai Jiao Tong University,  Shanghai 200240,  People's Republic of China.}
\email{xzhang@sjtu.edu.cn}


\date{}

\dedicatory{}

\keywords{Liouville integrability; Darboux integrability; Jacobian multiplier; Galois group.\\ \qquad To appear in {\it Transactions of the American Mathematical Society}}

\begin{abstract}
M.F. Singer [{\it Liouvillian first integrals of differential equations}, Trans. Amer. Math. Soc. 333 (1992), 673--688] proved the equivalence between  Liouvillian integrability and Darboux integrability for two dimensional polynomial differential systems. In this paper we will extend Singer's result to any finite dimensional polynomial differential systems. We prove that if an $n$--dimensional polynomial differential system has $n-1$ functionally independent Darboux Jacobian multiplier then it has $n-1$ functionally independent Liouvillian first integrals. Conversely if the system is Liouvillian integrable then it has a Darboux Jacobian multiplier.
\end{abstract}

\maketitle

\bigskip

\section{Background and statement of the main results}\label{s1}

The theory of integrability for differential systems is classic and it is useful in the study of dynamics of differential system. Integrability has different definitions in different fields. Here we mainly concern the algebraic aspects of integrability for polynomial differential systems, which involves analysis, algebraic geometry, the field extension and so on. For further information on this subject, we refer readers to Daboux~\cite{Da, Da1}, Jouanolou~\cite{Jo1979}, Prelle and Singer \cite{PS1983}, Singer \cite{Si1992}, Schlomiuk~\cite{Sc1993}, ¡¡Llibre~\cite{Ll2004},  Dumortier~and ~Llibre~{\it et al} ~\cite{DLA2006}, Christopher~{\it et al}~\cite{Chr1999, CLP2007} and Llibre~and~Zhang~\cite{LZ2009, LZ2012, PZ2013}.

Darboux theory of integrability was established by Darboux \cite{Da, Da1} in 1878 for polynomial differential systems of degree $n$ by using the invariant algebraic curves (resp. surfaces or hypersurfaces) in dimension 2 (resp. 3 or $n>3$). Jouanolou \cite{Jo1979} in 1979 extended the Darboux's theory to construct rational first integrals with the help of algebraic geometry. An elementary proof of Jouanolou's result was provided respectively by Christopher and Llibre \cite{CL2000} in 2000 for two dimensional differential systems and by Llibre and Zhang \cite{LZ2010BSM} in 2010 for any finite dimensional differential systems. On further extensions to Darboux theory of integrability, Christopher, Llibre and Pereira \cite{CLP2007} in 2007 took into account not only the number of invariant algebraic curves but also their multiplicities for two dimensional differential systems. Llibre and Zhang \cite{LZ2009} further extended Christopher {\it et al}'s result in \cite{CLP2007} to any finite dimensional differential systems, where there are some deep characterizations on the number of exponential factors and the multiplier of invariant algebraic hypersurfaces. Darboux theory of integrability has important applications in the center--focus problem, dynamical analysis and so on, see for instance \cite{CL2007,Ll2004,Sc1993,Zh2011} and the reference therein.

Darboux theory of integrability has a nice extension to Weierstrass integrability, see e.g. \cite{GG10}, where they used Weierstrass polynomials to replace the usual polynomials. By definition the former include the latter as a special one. In \cite{BP12} Bl\'{a}zquez--Sanz and Pantazi provided a new approach to study the Darboux integrability of polynomial differential systems of degree $m$, where they replaced the dimension of $\mathbb C_{m-1}[x]$ which the cofactors of Darboux polynomials and exponential factors are located in by the rank of a matrix associated to these cofactors. Here $\mathbb C_{m-1}[x]$ denotes the linear space formed by polynomials in $x\in\mathbb C^n$ of degree no more than $m-1$.  Recently Darboux theory of integrability was also successfully extended to nonautonomous differential systems which are polynomial ones in space variables with coefficients the smooth functions of the time, see e.g. \cite{LP09, GGL13}, where they extended the notion of invariant algebraic hypersurfaces in the phase space to polynomial invariant hypersurfaces in the extended space including the time.

Prelle and Singer \cite{PS1983} in 1983 proved that if a polynomial differential system has an elementary first integral then it has a first integral of a very simple form. As a corollary of their results, one gets that if a planar polynomial differential system has an elementary first integral, then it has an integrating factor of the form $f_1^{m_1}\ldots f_p^{m_p}$ with $f_i\in \mathbb C[x,y]$ and $m_i\in \mathbb Z$. This shows the equivalence between the existence of elementary first integrals and the Darboux integrating factors for planar polynomial differential systems.

Singer \cite{Si1992} in 1992 proved that a planar polynomial differential system has a Liouvillian first integral if and only if it has an integrating factor of the form
\[
R(x,y)=\exp\left(\int U(x,y)dx+V(x,y)dy\right),
\]
where $U(x,y),V(x,y)$ are rational functions in $x,y$. For a simple proof to Singer's result, see \cite{Chr1999} and \cite[Theorem 3.2]{CL2007}.

In this paper we will extend Singer's result to any finite dimensional polynomial differential systems.

Consider polynomial differential systems
\begin{equation}\label{2.1.1}
\dot x=P(x), \qquad x \in\mathbb C^n,
\end{equation}
where $P(x)=(P_1(x),\ldots,P_n(x))$~are vector--valued polynomial functions. We call $m:=\max\{\mbox{deg}P_1,\ldots,\mbox{deg}P_n\}$~\textit{ the degree of polynomial differential systems ~\eqref{2.1.1}}. In what follows we also use
\[
\mathcal X_P=P_1(x)\frac{\partial }{\partial x_1}+\ldots+P_n(x)\frac{\partial }{\partial x_n},
\]
to represent the vector field associated to system \eqref{2.1.1}. For simplifying notations, in what follows we denote ${\partial }/{\partial x_i}$ by $\partial_i$ and $(\partial_1,\ldots,\partial_n)$ by $\partial $.

Denote by $\mathbb C[x]$ the ring of polynomials in $x$ with coefficients in $\mathbb C$. A polynomial $f(x)\in\mathbb C[x]$ is called a {\it Darboux polynomial} of $\mathcal X_P$ if there exists a $k(x)\in \mathbb C[x]$ such that
\[
\mathcal X_P(f)(x)=k(x) \, f(x), \qquad x\in \mathbb C^n.
\]
The polynomial $k$ is called {\it cofactor} of $f$.
A function of the form
\[
\exp\left(\frac{g}{h}\right)f_1^{l_1}\ldots f_r^{l_r}, \quad \mbox{ with } g,\, h,\,  f_i\in \mathbb C[x],\,\, \, l_i\in \mathbb C,\,\, i=1,\ldots,r
\]
is called a {\it Darboux function}. For a Darboux function, we always require that its factors $f_i$ are irreducible and relatively different, and $g,h$ are relative coprime.

A {\it Darboux first integral} of \eqref{2.1.1} is a Darboux function and it is a first integral of \eqref{2.1.1}. Note that a first integral is not necessary to be defined in the full space but in a full Lebesgue measure subset of $\mathbb C^n$. System \eqref{2.1.1} is {\it Darboux integrable} if it has $n-1$ functionally independent Darboux first integrals.

A smooth function $J(x)$ is a {\it Jacobian multiplier} of system \eqref{2.1.1} if
\[
\partial_1(JP_1)+\ldots+\partial_n(JP_n)=0.
 \]
A {\it Darboux Jacobian multiplier} of system \eqref{2.1.1} is a Jacobian multiplier of the system and it is a Darboux function.
For planar polynomial differential systems, a Jacobian multiplier is usually called an {\it integrating factor}.
If a planar polynomial differential system has a Darboux integrating factor, it is also called  Darboux integrable.

For stating our results we recall the definition of Liouvillian functions.

A \textit{differential field } $(K,\,\Delta)$ consists of the field  $K$ and the set $\Delta$ of commutative derivatives defined on $K$. In this paper all mentioned fields have characteristic $0$.

A \textit{differential field extension} of a differential field $(K,\,\Delta)$ is a differential field $(L,\,\Delta')$ with the properties that $K\subset L$ and for  $\forall \, \delta'\in \Delta'$  we have $\left.\delta'\right|_{K}\in\Delta$. Because of the relation between the derivatives of differential field   $(K,\,\Delta)$ and its field extension $(L,\,\Delta')$, we also use $\Delta$ to represent $\Delta'$. For simplifying notations we also use $L/K$ to denote the differential field extension $(L,\Delta)$ of $(K,\Delta)$.

For a field extension $L/K$,
\begin{itemize}
\item $\alpha\in L$ is called
\begin{itemize}
\item \textit{an algebraic element of $K$}, if there exists a polynomial with coefficients in $K$ such that $F(\alpha)=0$.
\item a \textit{transcendental element of $K$}, if  $\alpha$ is not an algebraic element over $K$.
\end{itemize}

\item If each element of  $L$ is algebraic over $K$, we call  \textit{~$L/K$ an algebraic extension of field $K$}.
\item $L$ can be considered as a \textit{vector space over $K$}: the elements of $L$ are treated as vectors, and elements of $K$ are treated as scalars, and the summation of vectors is that of elements of field and the product of elements of $L$ and $K$ is that of elements of field $L$.
\begin{itemize}
\item The dimension of this vector space is called \textit{degree of this differential field extension}, denoted by $[L:K]$.
\end{itemize}
\item If  $[L:K]\in \mathbb N$, we call \textit{$L/K$ finite field extension}.
\item Let~$S\subset L$,
\begin{itemize}
\item $K(S)$ denotes the minimal subfield of $L$ including $K$ and $S$.
\item If $S$ contains only one element, we call  \textit{~$K(S)$ the minimal field extension of $K$}.
\end{itemize}
\end{itemize}

A differential field extension $L/K$ is \textit{Liouvillian}, if this differential field extension can be written in the tower form
\[
K=K_0\subset K_1\subset\ldots\subset K_r=L,
\]
such that
\begin{itemize}
\item[$(a)$] $K_{i+1}$ is a finite algebraic extension of $K_i$, or
\item[$(b)$] $K_{i+1}=K_i(t)$, where  $t$ is a transcendental element of $K_i$ satisfying: for each $\delta\in \Delta$, $\dfrac{\delta t}{t}\in K_i$, or
\item[$(c)$] $K_{i+1}=K_i(t)$, where $t$ is a transcendental element of $K_i$ satisfying: for $\delta\in \Delta$, $\delta t\in K_i$.
\end{itemize}

A {\it Liouvillian first integral} of \eqref{2.1.1} is a Liouvillian function and is a first integral of \eqref{2.1.1}. System \eqref{2.1.1} is {\it Liouvillian integrable} if it has $n-1$ functionally independent Liouvillian first integrals.

Now we can state our main results. The first one characterizes the existence of Liouvillian first integrals via Darboux Jacobian multipliers.

\begin{theorem}\label{th1}
If polynomial differential system \eqref{2.1.1} has $n-1$ functionally independent Darboux Jacobian multipliers, then they have $n-1$ functionally independent Liouvillian first integrals.
\end{theorem}

The next one shows that Liouvillian integrability implies the existence of Darboux Jacobian multipliers.

\begin{theorem}\label{th2}
If system \eqref{2.1.1} is Liouvillian integrable, i.e. it has $n-1$ functionally independent Liouvillian first integrals, then the system has
a Darboux Jacobian multiplier.
\end{theorem}

In the rest of this paper we will prove our main results.

\medskip
\section{Proof of the main results}\label{s2}

\setcounter{section}{2}
\setcounter{equation}{0}\setcounter{theorem}{0}

\subsection{Proof of Theorem \ref{th1}}
Let $J_1(x),\,\ldots,\, J_{n-1}(x)$ be $n-1$ functionally independent Darboux Jacobian multipliers of system \eqref{2.1.1}. Then we have
\begin{equation}\label{e1}
\mathcal X_P(J_l)=-J_l\,\mbox{div} P,\qquad l=1,\ldots, n-1,
\end{equation}
where $\mbox{div} P=\partial_1 P_1(x)+\ldots+\partial_n P_n(x)$ is the divergence of the vector fields $P(x)$. Recall that $\partial_i P_i$ denotes the partial derivative of the function $P_i$ with respect to $x_i$.

From the definition of Darboux functions and some direct calculations we get that
\[
\frac{\partial_i J_l}{J_l}\in \mathbb C(x),\qquad l\in\{ 1,\ldots, n-1\},\,\,\, i\in\{1,\ldots,n\}.
\]
Recall that $\mathbb C(x)$ is the field of rational functions in $x$. So it follows from the condition $(b)$ of Liouvillian extension of field that $J_l$ for $l=1,\ldots, n-1$ are Liouvillian functions. Furthermore, some easy calculations show that
\[
\frac{J_l}{J_k}, \qquad \mbox{for } \quad 1\le l\ne k\le n-1,
\]
are non--trivial Liouvillian first integrals of the vector field $\mathcal X_P$, i.e. $\frac{J_l}{J_k}$ is not a constant and $\mathcal X_P\left(\frac{J_l}{J_k}\right)\equiv 0$.

We claim that
\[
\frac{J_1}{J_{n-1}},\, \ldots, \, \frac{J_{n-2}}{J_{n-1}},
\]
are functionally independent. Indeed, assume that
\[
c_1\partial\left(\frac{J_1}{J_{n-1}}\right)+\ldots+c_{n-2}\partial\left(\frac{J_{n-2}}{J_{n-1}}\right)=0.
\]
Since
\[
\partial\left(\frac{J_{l}}{J_{n-1}}\right)=\frac{J_{n-1}\, \partial J_l-J_l\partial J_{n-1}}{J_{n-1}^2},
\]
we have
\begin{eqnarray*}
&& c_1J_{n-1}\partial J_1+\ldots+c_{n-2}J_{n-1}\partial J_{n-2}\\
&&\qquad \qquad \,\,-(c_1J_1+\ldots+c_{n-2}J_{n-2})\partial J_{n-1}=0.
\end{eqnarray*}
So by the functional independence of $J_1,\ldots, J_{n-1}$ we must have
\[
c_1J_{n-1}=\ldots=c_{n-2}J_{n-1}=c_1J_1+\ldots+c_{n-2}J_{n-2}=0,
\]
in a full Lebesgure measure subset of $\mathbb C^n$.
Consequently
\[
c_1=\ldots=c_{n-2}=0,
\]
in a full Lebesgure measure subset of $\mathbb C^n$. This proves the claim.

Using the last claim, we assume without loss of generality that
\begin{eqnarray*}
y_i&=& \frac{J_i}{J_{n-1}},\qquad i=1,\ldots,n-2,\\
y_{n-1}&=& x_{n-1},\\
y_n&=& x_n,
\end{eqnarray*}
are invertible, at least in some full Lebesgue measure subset $\Omega$ of $\mathbb C^n$. Denote by $y=G(x)$ this last transformation. Then under it
 the differential system \eqref{2.1.1} is equivalent to
\begin{eqnarray}\label{1.2.5}
\dot y_i&=&0, \quad\qquad i=1,\ldots,n-2,\nonumber\\
\dot y_{n-1}&=& P_{n-1}\circ G^{-1}(y),\\
\dot y_n&=& P_n\circ G^{-1}(y).\nonumber
\end{eqnarray}
Clearly system \eqref{1.2.5} has the first integrals $I_i(y)=y_i$, $i=1,\ldots, n-2$. In addition, we can prove that system \eqref{1.2.5} has the Jacobian multiplier
\[
M(y)=J_{n-1}\circ G^{-1}(y)D_yG^{-1}(y),
\]
where $D_yG^{-1}(y)$ denotes the Jacobian matrix of $G^{-1}$ with respect to $y$.
This shows that the two dimensional differential system
\begin{eqnarray*}
\dot y_{n-1}&=& P_{n-1}\circ G^{-1}(I_1,\ldots, I_{n-2},y_{n-1},y_n)=: g_{n-1}(y_{n-1},y_n),\\
\dot y_n&=& P_n\circ G^{-1}(I_1,\ldots, I_{n-2},y_{n-1},y_n)=: g_{n}(y_{n-1},y_n),
\end{eqnarray*}
has the integrating factor
\[
V(y_{n-1},y_n)=J_{n-1}\circ G^{-1}(y)D_yG^{-1}(I_1,\ldots,I_{n-2},y_{n-1},y_n),
\]
where we take $I_1,\ldots,I_{n-2}$ as constants. Hence this last two dimensional differential system has the first integral.
\[
I_{n-1}(y_{n-1},y_n)=\int V g_{n}dy_{n-1}-Vg_{n-1}dy_n.
\]
Obviously, $I_{n-1}$ is functionally independent of $I_1,\dots,I_{n-2}$, because the latter are independent of $y_{n-1}$ and $y_n$.

Next we prove that $I_{n-1}$ is a Liouvillian function. First we prove that $G^{-1}(x)$ is a Liouvillian function. Indeed, by
\[
G^{-1}\circ G(x)=x,
\]
we have
\[
D_y G^{-1}(G(x)) D_x G(x)=E.
\]
where $E$ is the $n$--dimensional unit matrix. Since $G$ is Liouvillian, and so is $(D_x G(x))^{-1}$. Hence we have
\[
D_y G^{-1}(G(x)) =(D_x G(x))^{-1},
\]
is a Liouvillian function. This shows that $G^{-1}(y)$ is a Liouvillian function. Furthermore it follows from the above construction that $g_{n-1},g_n$ and $V$ are Liouvillian functions. This proves that $I_{n-1}$ is a Liouvillian function.

Applying the transformation $y=G(x)$ to $I_1(y),\ldots,I_{n-1}(y)$, we get $n-1$ functionally independent Liouvillian first integrals
\[
H_1(x):=I_1\circ G(x),\quad \ldots, \quad H_{n-1}(x)=I_{n-1}\circ G(x),
\]
of differential system \eqref{2.1.1}. We complete the proof of the theorem.

\subsection{Proof of Theorem \ref{th2}}

For proving Theorem \ref{th2} and readers's convenience, we recall some notions.

Given a field $K$,
\begin{itemize}
\item A \textit{ separating field} of a polynomial $p(x)$ over  $K$ is a minimal field extension of $K$ such that $p(x)$ can be decomposed into product of  linear factors over this field extension, i.e.  $p(x)=\prod (x-a_i)$, $a_i\in L$, $L/K$ is the minimal field extension such that this decomposition can happen.
\item We say that an algebraic field extension $L/K$ of $K$ is \textit{normal}, if $L$ is a separating field of polynomials in $K[x]$.
\item The \textit{normal closure} of an algebraic field extension $L/K$ is a field extension $\overline L$ of $L$ such that $\overline L/K$ is normal, and  $\overline L$ is the minimal field extension satisfying this property.
\item \textit{Field automorphism} over field $K$ is a bijective map $\varphi:\,\, K\rightarrow K$ which keeps the algebraic properties of $K$, i.e. $\varphi$ satisfies that $\varphi(0_K)=0_K$, $\varphi(1_K)=1_K$, $\varphi(a+b)=\varphi(a)+\varphi(b)$ and $\varphi(ab)=\varphi(a)\varphi(b)$.
\item The set of all field automorphisms over field  $K$ fixing elements of a subfield $K'\subset K$ forms a group under the composition of maps. This group   is called \textit{Galois group}.
\item The \textit{order of a group} is the number of elements of a group $G$, denoted by~$|G|$.
\end{itemize}

Now we can prove Theorem \ref{th2}. It will follows from the following lemmas.

\begin{lemma}\label{lth21}
If system \eqref{2.1.1} is Liouvillian integrable, then it has a Jacobian multiplier of the form
\[
J=\exp\left(\int U_1dx_1+\ldots+U_ndx_n\right),
\]
with $U_i\in \mathbb C(x)$, $i=1,\ldots,n$, and
\[
\partial_iU_j=\partial_jU_i,\qquad 1\le j<i\le n.
\]
\end{lemma}

\begin{proof}
Assume that system \eqref{2.1.1} has the functionally independent Liouvillian first integrals $H_1,\ldots,H_{n-1}$, which are defined in a full Lebesgue measure subset $\Omega$ of $\mathbb C^n$.
By definition of first integrals we have
\begin{equation}\label{e3}
\mathcal X_P(H_i)(x)\equiv 0, \quad  x\in\Omega, \qquad i=1,\ldots,n-1.
\end{equation}

From independence of $H_1,\ldots, H_{n-1}$ we can assume without loss of generality that
\[
\Gamma:=\det\left(\partial_1\mathcal H,         \cdots ,  \partial_{n-1}\mathcal H\right)\ne 0,\qquad x\in \Omega,
\]
with
\[
\mathcal H:=(H_1,\ldots,H_{n-1})^T,
\]
where $T$ denotes the transpose of a matrix, and
\[
\partial_i\mathcal H:=(\partial_iH_1,\ldots,\partial_iH_{n-1})^T,\qquad i=1,\ldots,n.
\]
Set for $i=1,\ldots,n-1$
\[
\Gamma_i:=\det\left(
\partial_1\mathcal H,   \cdots,    \partial_{i-1}\mathcal H ,     \partial_n \mathcal H,  \partial_{i+1}\mathcal H ,    \cdots, \partial_{n-1}\mathcal H\right).
\]
Clearly $\Gamma$ and $\Gamma_i$, $i=1,\ldots,n$, are Liouvillian functions.

By the Cramer's rule we get from \eqref{e3} that
\[
P_i(x)=-\frac{\Gamma_i}{\Gamma} P_n(x), \quad i=1,\ldots,n-1.
\]
Hence we have
\[
\Gamma (P_1(x),\ldots,P_{n-1})=-(\Gamma_1,\ldots, \Gamma_{n-1})P_n(x).
\]
Since $P_1,\ldots,P_{n-1},P_n$ are relative coprime, and $\Gamma$ and $\Gamma_i$ are Liouvillian functions, so there exists a Liouvillian function $h(x)$ such that
\[
h(x)\Gamma=P_n(x).
\]
Consequently we have
\[
P_i(x)=-h(x)\Gamma_i,\quad i=1,\ldots,n-1.
\]
Set
\begin{equation}\label{Ai1}
A_i:=\frac{\partial_i h}{h}, \quad i=1,\ldots,n;  \qquad A:=(A_1,\ldots, A_n).
\end{equation}
Then $A_i$ is Liouvillian for each $i\in\{1,\ldots,n\}$.

We claim that
\begin{eqnarray}
\partial_iA_j&=&\partial_jA_i,\qquad 1\le j<i\le n, \label{e4}\\
\langle A,P\rangle&:=& A_1P_1+\ldots +A_nP_n=\mbox{div} P. \label{e5}
\end{eqnarray}
The equality \eqref{e4} can be proved by direct calculations via the fact
\[
\partial_j\partial_ih=\partial_i\partial_j h\qquad \mbox{ for all } \,\,\, i,j\in\{i,\ldots,n\}.
\]
For proving \eqref{e5}, some computations show that
\begin{eqnarray*}
\mbox{div}P&=&\partial_1(-h\Gamma_1)+\ldots+\partial_{n-1}(-h\Gamma_{n-1})+\partial_n(h\Gamma)\\
&=& A_1P_1+\ldots +A_nP_n+h(\partial_n\Gamma-\partial_1\Gamma_1-\ldots-\partial_{n-1}\Gamma_{n-1}).
\end{eqnarray*}
Next we only need to prove that
\[
\partial_n\Gamma-\partial_1\Gamma_1-\ldots-\partial_{n-1}\Gamma_{n-1}=0.
\]
Since
\begin{eqnarray*}
\partial_n\Gamma&=&\sum\limits_{i=1}\limits^{n-1}\det\left(\partial_1 \mathcal H, \ldots, \partial_{i-1} \mathcal H, \partial_n\partial_i \mathcal H, \partial_{i+1} \mathcal H, \ldots, \partial_{n-1} \mathcal H\right),\\
\partial_i\Gamma_i&=&\det \left(\partial_1 \mathcal H, \ldots, \partial_{i-1} \mathcal H, \partial_i\partial_n \mathcal H, \partial_{i+1} \mathcal H, \ldots, \partial_{n-1} \mathcal H\right)\\
&&+\sum\limits_{j=1,j\ne i}\limits^{n-1}\det\left(\partial_1 \mathcal H, \ldots, \partial_{j-1} \mathcal H, \partial_i\partial_j \mathcal H, \partial_{j+1} \mathcal H, \ldots, \partial_{i-1} \mathcal H, \right.\\
&& \qquad\qquad\qquad\qquad\qquad\left. \partial_n \mathcal H, \partial_{i+1}\mathcal H,\ldots,\partial_{n-1}\mathcal H\right),
\end{eqnarray*}
we have
\begin{eqnarray*}
&&\partial_1\Gamma_1+\ldots+\partial_{n-1}\Gamma_{n-1}-\partial_n\Gamma\\
&=& \sum\limits_{i=1}\limits^{n-1}\sum\limits_{j=1,j\ne i}\limits^{n-1}\det\left(\partial_1 \mathcal H, \ldots, \partial_{j-1} \mathcal H, \partial_i\partial_j \mathcal H, \partial_{j+1} \mathcal H, \ldots, \partial_{i-1} \mathcal H, \right.\\
&& \qquad\qquad\qquad\qquad\qquad\quad\left. \partial_n \mathcal H, \partial_{i+1}\mathcal H,\ldots,\partial_{n-1}\mathcal H\right)\\
&=&0,
\end{eqnarray*}
where in the last equality we have used the fact that
\begin{eqnarray*}
&&\det\left(\partial_1 \mathcal H, \ldots, \partial_{j-1} \mathcal H, \partial_i\partial_j \mathcal H, \partial_{j+1} \mathcal H, \ldots, \partial_{i-1} \mathcal H,\right. \\
&& \qquad\qquad\qquad\qquad\qquad\quad \quad \left. \partial_n \mathcal H, \partial_{i+1}\mathcal H,\ldots,\partial_{n-1}\mathcal H\right)\\
&&\qquad +\det\left(\partial_1 \mathcal H, \ldots, \partial_{j-1} \mathcal H, \partial_n \mathcal H, \partial_{j+1} \mathcal H, \ldots, \partial_{i-1} \mathcal H, \right. \\
&& \qquad\qquad\qquad\qquad\qquad\quad \quad \left.\partial_j\partial_i \mathcal H, \partial_{i+1}\mathcal H,\ldots,\partial_{n-1}\mathcal H\right)=0.
\end{eqnarray*}
This proves the claim.

Set
\[
U_i=-A_i,\qquad i=1,\ldots,n,
\]
with $A_i$ defined in \eqref{Ai1}. By \eqref{e4} and the Stokes's theorem (see e.g. \cite[p.3]{BT82}), we get that
\begin{equation}\label{JaJ1}
J=\exp\left(\int U_1dx_1+\ldots+U_ndx_n\right),
\end{equation}
is well defined in any connected subsets where $U_1,\ldots,U_n$ are defined. Furthermore we can check that $J$ in \eqref{JaJ1}
is a Jacobian multiplier of system \eqref{2.1.1} if and only if \eqref{e5} holds. So, in what follows we only need to prove that there exist rational functions $B_i$ for $i=1,\ldots,n$ instead of  $A_i$ such that the equalities \eqref{e4} and \eqref{e5} hold.

According to the Liouvillian extension of field in the tower form, all the $A_i$ belong to some tower for $i=1,\ldots,n$. We distinguish three different
cases according to the definition of the tower.

\noindent{$(a)$} $A_i\in K_{l+1}$, $i=1,\ldots, n$, and $K_{l+1}$ is an algebraic extension of the field $K_l$. We will prove that there exist $B_i\in K_l$ instead of $A_i$ for $i=1,\ldots,n$ such that \eqref{e4} and \eqref{e5} hold.

Let $\overline K_{l+1}$ be the normal closure of $K_{l+1}$, and $\mathcal G$ be the Galois group formed by the automorphisms of $\overline K_{l+1}$ fixing $K_l$.
Then it follows from a result of Artin (see Lang \cite[Theorem 1.1]{La2002}) that $\mathcal G$ is of finite order, and denote by $N=|\mathcal G|$ the order of the group. Note that $N\le [K_{l+1}:K_l]$, the degree of the algebraic extension of the field.

Since $P\in \mathbb C(x)^n$ and $\mathbb C(x)\subset K_l$, we get from \eqref{e4} and \eqref{e5} that
\begin{equation}\label{e6}
\displaystyle\begin{array}{rcl}
\partial_ig(A_j)&=&\partial_jg(A_i),\qquad \forall g\,\in\mathcal G,\\
\left\langle \sum\limits_{g\in \mathcal G}g(A),P\right\rangle&=&\sum\limits_{g\in \mathcal G}g(A_1)P_1+\ldots+\sum\limits_{g\in \mathcal G}g(A_n)P_n = N\mbox{div}P,
\end{array}
\end{equation}
where in the second equality we have used the fact that $g\in\mathcal G$ fixes $K_l$. Set
\[
B_i=\frac 1N\sum\limits_{g\in \mathcal G}g(A_i),\qquad i=1,\ldots,n.
\]
Then $B_i\in K_l$ for $i=1,\ldots,n$, because all $B_i$ are fixed under the action of all elements of the Galois group $\mathcal G$. Furthermore, we get from \eqref{e6} that \eqref{e4} and \eqref{e5} hold for $B_i$ instead of $A_i$ for $i=1,\ldots,n$.

\smallskip

\noindent $(b)$ Assume $K_{l+1}=K_l(t)$ with $t$ a transcendental element over $K_{l}$ and $\partial_it/t\in K_l$ for $i=1,\ldots,n$.

Since $A_i\in K_l(t)$ for $i=1,\ldots,n$, we can assume without loss of generality that
\[
A_i=a_i(t)\in K_l(t),\qquad i=1,\ldots,n.
\]
Expanding $a_j(t)$ into Laurent series in $t$ gives
\begin{equation}\label{e7}
a_j(t)=a_0^{(j)}+\sum\limits_{s\in\mathbb Z\setminus\{0\}}a_s^{(j)}t^s,\qquad j=1,\ldots, n,
\end{equation}
with $a_s^{(j)}\in K_l$ for $j=1,\dots,n$ and all $s$. Then we have
\begin{equation}\label{e8}
\partial_iA_j=\partial_ia_0^{(j)}+\sum\limits_{s\in\mathbb Z\setminus\{0\}}\left(\partial_ia_s^{(j)}+sa_s^{(j)} p_i  \right)t^s,
\end{equation}
where $p_i\in K_l$ satisfying $\partial_it/t=p_i\in K_l$ for $i=1,\ldots,n$. Set
\[
B_i=a_0^{(i)}, \qquad i=1,\ldots,n,
\]
we have $B_i\in K_{l}$. Substituting \eqref{e7} and \eqref{e8} into \eqref{e4} and \eqref{e5}, and equating the coefficients of $t^0$, we get that $B_i$ for $i=1,\ldots,n$ satisfy \eqref{e4} and \eqref{e5} instead of $A_i$.

\noindent $(c)$ Assume that $K_{l+1}=K_l(t)$ with $t$ a transcendental element over $K_l$ and $\partial_i t\in K_l$ for all $i\in\{1,\ldots,n\}$.

Similar to $(b)$ we set
\[
A_j=a_j(t)\in K_l(t),\qquad j=1,\ldots,n.
\]
Now the Laurent expansion in $t$ does not work, we choose the Laurent expansion in $1/t$ of $a_j(t)$. Since $A_i\in K_l(t)$,
we get from its construction \eqref{Ai1}, i.e. $A_i=\partial_ih/h$, that the degree of numerator in $t$ of $a_j(t)$ is less than or equal to the degree of its denominator.
Write $a_j(t)$ as
\begin{eqnarray*}
a_j(t)&=&\frac{a_{j0}+a_{j1}t+\ldots+a_{jk}t^k}{b_{j0}+b_{j1}t+\ldots+b_{jl}t^l}\\
& = &
\frac{a_{j0}t^{-l}+a_{j1}t^{-(l-1)}+\ldots+a_{jk}t^{-(l-k)}}{b_{j0}t^{-l}+b_{j1}t^{-(l-1)}+\ldots+b_{jl}}.
\end{eqnarray*}
Since $l\ge k$, so the Laurent expansion in $t^{-1}$ of $a_j(t)$ has the form
\begin{equation}\label{e9}
a_j(t)= \sum\limits_{s=-\infty}\limits^{0}a_s^{(j)}t^s,\qquad j=1,\ldots, n,
\end{equation}
with $a_s^{(j)}\in K_l$ for $j=1,\dots,n$ and all $s$. Direct calculation shows that
\begin{equation}\label{e10}
\partial_iA_j=\sum\limits_{s=-\infty}\limits^{0}\left(\partial_ia_{s-1}^{(j)}+sa_s^{(j)}q_i\right)t^{s-1}+\partial_ia_0^{(j)},
\end{equation}
where $q_i\in K_l$ satisfying $q_i=\partial_it\in K_l$ for $i=1,\ldots,n$.

Set
\[
B_i=a_0^{(i)}, \qquad i=1,\ldots,n.
\]
Then $B_i\in K_{l}$. Using the expansions \eqref{e9} and \eqref{e10}  we get from \eqref{e4} and \eqref{e5} that
\begin{eqnarray*}
\partial_iB_j&=&\partial_jB_i,\qquad 1\le j<i\le n,\\
\mbox{div} P&=&  B_1P_1+\ldots +B_nP_n .
\end{eqnarray*}
Of course if all $B_i=0$, then $\mbox{div}P=0$, and so $J=1$ is a Jacobian multiplier.

Summarizing the cases $(a)$, $(b)$ and $(c)$, and combining the definition of Liouvillian functions in the tower form, by induction we get that there exist
$U_1,\ldots,U_n\in K_0=\mathbb C(x)$ for which \eqref{e4} and \eqref{e5} hold instead of $A_1,\ldots,A_n$. We complete the proof of the lemma.
\end{proof}

The next result shows that the existence of Jacobian multipliers of the form given in Lemma \ref{lth21} implies the existence of Darboux Jacobian multipliers.

\begin{lemma}\label{l2.2.2}
If polynomial differential system \eqref{2.1.1} has a Jacobian multiplier
\[
J=\exp\left(\int U_1dx_1+\ldots+U_ndx_n\right),
\]
with $U_i\in \mathbb C(x)$, and $\partial_jU_i=\partial_iU_j$ for $1\le i,j\le n$, then it has a Darboux Jacobian multiplier
\[
\exp\left(\frac gh\right)\prod\limits_if_i^{l_i},
\]
where $g,h,f_i\in \mathbb C[x,y]$, $l_i\in \mathbb C$.
\end{lemma}

\begin{proof}
Since $U_1,\ldots,U_n\in \mathbb C(x)$, we treat their numerators and denominators as polynomials in $x_1$ with coefficients in $\mathbb C[x_2,\ldots,x_n]$. Let $K$ be the minimal normal algebraic field extension of $\mathbb C(x_2,\ldots,x_n)$ such that it is the separating field of the numerators and denominators of $U_1,\ldots,U_n$.

By the properties on normal algebraic field extension, the rational functions  $U_1,\ldots, U_n$ over $K$ can be expanded in
\begin{equation}\label{zx627}
U_k(x)= \sum\limits_{i=1}\limits^r\sum\limits_{j=1}\limits^{m}\frac{\alpha_{ij}^{(k)}}{(x_1-\beta_i)^j}
        +\sum\limits_{i=0}\limits^{p}\xi_i^{(k)}x_1^i,\qquad k=1,\ldots,n,
\end{equation}
where $\alpha_{ij}^{(k)}, \beta_i, \xi_i^{(k)}\in K$, and parts of them can be zero. Direct calculations show that  for $l\in\{2,\ldots,n\}$
\begin{eqnarray*}
\partial_lU_1 &=&\sum\limits_{i=1}\limits^r\sum\limits_{j=1}\limits^{m}\left(\frac{\partial_l\alpha_{ij}^{(1)}}{(x_1-\beta_i)^j}
              +\frac{j\alpha_{ij}^{(1)}\partial_l\beta_i}{(x_1-\beta_i)^{j+1}}\right) +\sum\limits_{i=0}\limits^{p}\partial_l\xi_i^{(1)}x^i,\\
\partial_1U_l &=& \sum\limits_{i=1}\limits^r\sum\limits_{j=1}\limits^{m}\frac{-j\alpha_{ij}^{(l)}}{(x_1-\beta_i)^{j+1}}
             +\sum\limits_{i=0}\limits^{p}i\xi_{i}^{(l)}x_1^{i-1}.
\end{eqnarray*}
Using the assumption $\partial_1U_l=\partial_lU_1$, and comparing the coefficients of $(x_1-\beta_{i})^{-j}$ and $x^i$, we get that  for $l\in\{2,\ldots,n\}$
\begin{equation}\label{2.2.2}
\partial_l\alpha_{i,j+1}^{(1)}+j\alpha_{ij}^{(1)}\partial_l\beta_i+j\alpha_{ij}^{(l)}=0,\qquad \partial_l\xi_i^{(1)}=(i+1)\xi_{i+1}^{(l)}.
\end{equation}
The first equality with $j=0$ of \eqref{2.2.2} shows that $\alpha_{i1}^{(1)}\in \mathbb C$, because $\alpha_{ij}^{(k)}$ are functions in $x_2,\ldots,x_n$ as prescribed.

Set
\begin{eqnarray*}
\Phi(x)&=&\sum\limits_{i=1}\limits^r\alpha_{i1}^{(1)}\log (x_1-\beta_i)+\sum\limits_{i=1}\limits^r\sum\limits_{j=2}\limits^{m}\frac{-1}{j-1}\frac{\alpha_{ij}^{(1)}}{(x_1-\beta_i)^{j-1}}\\
&&          +\sum\limits_{i=1}\limits^p\frac{\xi_i^{(1)}}{i+1}x_1^{i+1}+\int \xi_0^{(2)} dx_2+\ldots+\xi_0^{(n)} dx_n,
\end{eqnarray*}
where the integration represents any primitive function of $\xi_0^{(2)} dx_2+\ldots+\xi_0^{(n)} dx_n$. Direct calculations show that $\partial_1\Phi=U_1$.
For $l>1$ since
\begin{eqnarray*}
\partial_l\Phi(x)&=&-\sum\limits_{i}\frac{\alpha_{i1}^{(1)}\partial_l\beta_i}{x_1-\beta_i}
 +\sum\limits_{i,j}\frac{-1}{j-1}\left(\frac{\partial_l\alpha_{ij}^{(1)}}{(x_1-\beta_i)^{j-1}}+\frac{(j-1)\alpha_{ij}^{(1)}\partial_l\beta_i}{(x_1-\beta_i)^j}\right)
          \\
          && +\sum\limits_{i}\frac{\partial_l\xi_i^{(1)}}{i+1}x^{i+1}+\xi_0^{(l)}.
\end{eqnarray*}
Using the equalities \eqref{2.2.2} and by some calculations, we get that
\[
\partial_l\Phi=U_l, \qquad l=2,\ldots,n.
\]
These show that
\[
\Phi(x)=\int U_1dx_1+\ldots+U_ndx_n,
\]
with possible a constant difference.

Denote by $\mathcal G$ the group of automorphisms over $K$ which keep $\mathbb C(y)$, where $y=(x_2,\ldots,x_n)$. Since $K$ is the minimal normal algebraic field extension of $\mathbb C(y)$, we get from the properties of field extensions that $\mathcal G$ is a finite group. Denote by $N=|\mathcal G|$, the order of $\mathcal G$.
Set
\[
\Psi=\frac 1N\sum\limits_{\sigma\in\mathcal G}\sigma(\Phi).
\]
Since $\sigma\in\mathcal G$ is an automorphism over the algebraic field extension $K$ of $\mathbb C(y)$, it follows that
\begin{eqnarray*}
\sigma\left(\alpha_{i1}\log (x_1-\beta_i)\right)&=&\alpha_{i1}\log (x_1-\sigma(\beta_i)), \\
 \sigma\left(\int \gamma_0^{(2)} dx_2+\ldots+\gamma_0^{(n)} dx_n\right)&=& \int\sigma(\gamma_0^{(2)})dx_2+\ldots+\sigma(\gamma_0^{(n)})dx_n,
\end{eqnarray*}
where the second equality may have a constant difference.

Since  $\sigma(\partial_l\Phi)=\sigma(U_l)$, we have
\begin{equation}\label{xxx1}
  \partial_l\sigma(\Phi)=\sigma(U_l)=U_l,\qquad l=1,\ldots,n,
\end{equation}
where in the second equality we have used the facts that the numerators and denominators of $U_l$'$s$ can be written in the polynomials of $x_1$ with
coefficients in $\mathbb C(y)$ and $\sigma$ keeps $\mathbb C(y)$.

The  equalities \eqref{xxx1} show that
\[
\partial_l\Psi=U_l,\qquad   l=1,\ldots,n.
\]
Moreover we have
\[
\Psi(x)=\sum\limits_{i=1}\limits^{r_0}c_i\log R_i(x)+R(x)+\int S_2(y)dx_2+\ldots+S_n(y)dx_n,
\]
where $c_i\in\mathbb C$, $R_i,\,R\in\mathbb C(x)$ and $S_i\in\mathbb C(y)$. Recall that $y=(x_2,\ldots,x_n)$.
By the expansions of $U_k$'$s$ in \eqref{zx627} and $\partial_lU_k=\partial_kU_l$, it follows that
\[
\partial_l\xi_0^{(k)}(y)=\partial_k\xi_0^{(l)}(y),\qquad 2\le k,l\le n.
\]
So we have
\[
\sum\limits_{\sigma\in\mathcal G}\sigma(\xi_0^{(k)})\in\mathbb C(y),\qquad k=2,\ldots,n,
\]
and for $2\le k,l\le n$
\[
\partial_l\left(\sum\limits_{\sigma\in\mathcal G}\sigma(\xi_0^{(k)})\right)=
\sum\limits_{\sigma\in\mathcal G}\sigma(\partial_l\xi_0^{(k)})=
\sum\limits_{\sigma\in\mathcal G}\sigma(\partial_k\xi_0^{(l)})=
\partial_k\left(\sum\limits_{\sigma\in\mathcal G}\sigma(\xi_0^{(l)})\right).
\]
These show that
\[
\partial_{x_i}S_j(y)=\partial_{x_j}S_i(y), \quad   2\le i,j\le n.
\]

Now for the integration $\int S_2(y)dx_2+\ldots+S_n(y)dx_n$ with $y=(x_2,\ldots,x_n)$ we are in the same conditions as those of integration
$\int U_1(y)dx_1+\ldots+U_n(y)dx_n$, so working in a similar way as that in the above proof we get that there exists a function $\Psi_1(y)$ such that
\[
\partial_l\Psi_1(y)=S_l(y), \qquad l=2,\ldots,n,
\]
and
\[
\Psi_1(y)=\sum\limits_{i=1}\limits^{r_1}d_i\log T_i(y)+T(y)+\int W_3(z)dx_3+\ldots+W_n(z)dx_n,
\]
where $d_i\in\mathbb C$, $T_i,\,T\in\mathbb C(y)$, and $W_i\in\mathbb C(z)$ with $z=(x_3,\ldots,x_n)$ satisfy
\[
\partial_{x_i}W_j(z)=\partial_{x_j}W_i(z),\qquad 3\le i,\, j\le n.
\]

By induction we can prove that
\begin{eqnarray*}
\Psi(x)&=&\sum\limits_{i=1}\limits^{r_0}c_i\log R_i(x)+R(x)\\
&& +\sum\limits_{j=2}^n\left(\sum\limits_{i=1}\limits^{r_j}c_i^{(j)}\log R_i^{(j)}(x_j,\ldots,x_n)+R^{(j)}(x_j,\ldots,x_n)\right),
\end{eqnarray*}
where $c_i,c_i^{(j)}\in \mathbb C$, and $R_i^{(j)},R^{(j)}\in \mathbb C(x_j,\ldots,x_n)$.  Recall that $\partial_l\Psi(x)=U_l(x)$ for $l=1,\ldots,n$.
Furthermore we have
\begin{eqnarray*}
\exp\left(\Psi\right)&=&\exp\left(R(x)+\sum\limits_{j=2}^n R^{(j)}(x_j,\ldots,x_n)\right)\\
&&\times \prod\limits_{i=1}\limits^{r_0} \left(R_i(x)\right)^{c_i}\prod\limits_{j=2}^n\prod\limits_{i=1}^{r_j}\left(R_i^{(j)}(x_j,\ldots,x_n)\right)^{c_i^{(j)}}.
\end{eqnarray*}
Since $R,\, R^{(j)},\, R_i,\, R_i^{(j)}\in \mathbb C(x)$, it follows that $\exp(\Psi(x))$ is a Darboux function, and consequently is a Darboux Jacobian multiplier. This proves Lemma  \ref{l2.2.2}.
\end{proof}

\smallskip

Summarizing Lemmas  \ref{lth21} and \ref{l2.2.2}, we complete the proof of Theorem \ref{th2}.

\medskip


\noindent{\bf Acknowledgements.} The author is partially supported by
NNSF of China grant 11271252,  RFDP of Higher Education of China grant 20110073110054, and FP7-PEOPLE-2012-IRSES-316338 of Europe.

\bigskip
\bigskip
\bibliographystyle{amsplain}

\end{document}